\documentclass[12pt,a4]{amsart}

\calclayout

\usepackage{xcolor}
\usepackage[lite]{amsrefs}
\usepackage{amssymb}
\usepackage[all,cmtip]{xy}

\usepackage{mathrsfs}
\usepackage{tabularx}
\usepackage{booktabs}
\usepackage[labelfont=bf,format=plain,justification=raggedright,singlelinecheck=false]{caption}
\usepackage{bookmark}
\usepackage{hyperref}

\newcommand{\F}{\mathbb{F}}

\newcommand{\aut}{\mathrm{Aut}}

\numberwithin{equation}{section}

\theoremstyle{plain}
\newtheorem{theorem}[equation]{Theorem}
\newtheorem{corollary}[equation]{Corollary}
\newtheorem{lemma}[equation]{Lemma}

\theoremstyle{definition}

\theoremstyle{remark}
\newtheorem{remark}[equation]{Remark}

\newcommand{\mi}{\alpha}

\begin{document}

\title{A maximal function field of genus 17 over $\mathbb{F}_{11^2}$.}
\author{Peter Beelen}
\address{Department of Applied Mathematics and Computer Science, Technical University of Denmark, DK-2800, Kongens Lyngby, Denmark}
\email{pabe@dtu.dk}

\author{Maria Montanucci}
\address{Department of Applied Mathematics and Computer Science, Technical University of Denmark, DK-2800, Kongens Lyngby, Denmark}
\email{marimo@dtu.dk}

\author{Jonathan Niemann}
\address{Department of Applied Mathematics and Computer Science, Technical University of Denmark, DK-2800, Kongens Lyngby, Denmark}
\email{jtni@dtu.dk}

\keywords{Hermitian function field, maximal function field, automorphism group}
\subjclass[2010]{Primary 11G, 14G, 14H05} 

\begin{abstract}
In this article, we describe a new maximal function field $G$ over the finite field $\F_{11^2}$ with $11^2$ elements and show that it cannot be obtained as a subfield of the Hermitian function field. 
This provides the first known $\mathbb{F}_{p^2}$-maximal function field that satisfies this property. Further, we compute the automorphism group of $G$. The new maximal function field was found using AI.
\end{abstract}

\maketitle

\section{Introduction}

Let $\F_q$ denote the finite field with $q$ elements and denote by $F$ a function field with full constant field $\F_q$ and transcendence degree one over $\F_q$. Such function fields arise in a natural way as the function field of an absolutely irreducible algebraic curve defined over $\F_q$. Let $N_1(F)$ (resp. $g(F)$) denote the number of $\F_q$-rational places of $F$ (resp. the genus of $F$). The famous Hasse--Weil bound states that 
$$N_1(F) \le q+1+2\sqrt{q}g(F).$$
Function fields $F$ of positive genus attaining the Hasse--Weil bound are called maximal function fields. Note that for maximal function fields, necessarily $q$ is a square.

A famous example of a maximal function field is the Hermitian function field $H_q:=\F_{q^2}(x,y)$, where $x^{q+1}+y^{q+1}+1=0$. According to a result by Serre, see for example \cite[Proposition 6]{Kleiman}, any subfield of $H_q$ of transcendence degree one over $\F_{q^2}$ having full constant field $\F_{q^2}$, is again maximal. For a time it was believed that perhaps all maximal function fields could be obtained in this way, but Giulietti and Korchm\'aros found a maximal function field with constant field $\F_{q^6}$ in \cite{giuliettiNewFamilyMaximal2009} that is not isomorphic to a subfield of $H_{q^3}$. Further known maximal function fields which are not subfields of $H_{q^3}$ were found in \cite{GQZ,BDM, TTT} as subfields of the Giulietti and Korchm\'aros function field. On the other hand, for $q=p$ a prime number, all known maximal function fields with constant field $\F_{p^2}$ are subfields of the Hermitian function field $H_p$. 
In fact it has been shown in \cite{BMT} that every $\mathbb{F}_{p^2}$-maximal function field of genus $g$ that has more than $84(g-1)$ automorphisms is a Galois subfield of the Hermitian function field.
The main result of this paper is that for $p=11$, there exists a maximal function field $G$ over $\F_{11^2}$ that is not isomorphic to a subfield of $H_{11}$. 

The paper is organized as follows: we start with introducing the function field $G$ and proving its maximality. Next we compute the automorphism group of $G$. Finally, we show that $G$ is not isomorphic to a subfield of $H_{11}$.

\section{Proof of maximality}

The function field $\F_{11^2}(x,y)$ with $(y^2-1)^2=1-x^{12}$, is a known subfield of the Hermitian curve $H_{11}$. In fact, after replacing $y$ by $3y'$ and $x$ by $2x'$ it is clear that it is of the type described in \cite[Example 6.4, Case 2]{garciaSubfieldsHermitianFunction2000}. It follows from their work that it is a maximal function field of genus nine. 

Now let $\mi \in \F_{11^2}$ be an element satisfying $\mi^2 = -1$ and define the polynomial $f(X) = (X+1)(X^2+X+1)\in \F_{11^2}[X]$. The function field of main interest in this article is the function field $G = \F_{11^2}(x,y,z)$ defined by the equations
\begin{align*}
    \begin{cases}
        (y^2-1)^2 = 1 - x^{12}, \\
        z^2 = (\mi+1)f(x)f(\mi x).
    \end{cases}
\end{align*}

\begin{lemma}\label{lem:genusG}
The function field $G$ has genus $17$.
\end{lemma}
\begin{proof}
We have already mentioned that the function field $\F_{11^2}(x,y)$ is a known maximal function field of genus nine. In particular, it has $320$ $\F_{11^2}$-rational places.
Note that the extension $\F_{11^2}(x,y)/\F_{11^2}(x)$ can be obtained as two successive Kummer extensions of degree two, namely $\F_{11^2}(x) \subset \F_{11^2}(x,s) \subset \F_{11^2}(x,y)$, where $s=y^2-1$ (and hence $s^2=1-x^{12}$). The theory of Kummer extensions implies that the only ramified places in the extension $\F_{11^2}(x,s)/\F_{11^2}(x)$ are the zeroes of $x^{12}-1$. Then the Riemann--Hurwitz theorem implies that the genus of $\F_{11^2}(x,s)$ is five. Since we already know that the genus of $\F_{11^2}(x,y)$ is nine, this implies that the extension $\F_{11^2}(x,y)/\F_{11^2}(x,s)$ is unramified. 
We conclude that only the zeroes $x^{12}-1$ ramify in $\F_{11^2}(x,y)/\F_{11^2}(x)$ and that all ramification indices are equal to two.

The extension $\F_{11^2}(x,z)/\F_{11^2}(x)$ is a Kummer extension of degree two. It is ramified precisely in the zeroes of $f(x)f(\mi x)$ and the pole of $x$. Of course all ramification indices are two in this case. Also note that $x^{12}-1=-f(x)f(\mi x)f(-x)f(-\mi x)$. In particular, by the above, any place $P$ that ramifies in the extension $\F_{11^2}(x,z)/\F_{11^2}(x)$, will also ramify in the extension $\F_{11^2}(x,y)/\F_{11^2}(x)$, in both cases with ramification index two. Using Abhyankhar's lemma \cite[Theorem 3.9.1]{stichtenothAlgebraicFunctionFields2009}, we can therefore conclude that the extension $\F_{11^2}(x,y,z)/\F_{11^2}(x,y)$ is unramified. Moreover, the extension degree $[\F_{11^2}(x,y,z):\F_{11^2}(x,y)]$ is two. Indeed, assume for the moment that $\F_{11^2}(x,y,z)=\F_{11^2}(x,y)$. One of the zeroes of $x$ in $\F_{11^2}(x,y)$ is the $\F_{11^2}$-rational place determined by the values $s(Q)=1$ and $y(Q)=\sqrt{2}$. However, we find that $z(Q)^2=(\mi+1)f(0)^2=\mi+1$, which is not a square in $\F_{11^2}$. In particular, any place of $\F_{11^2}(x,y,z)$ lying above $Q$ has residue field $\F_{11^4}$. But this is impossible if $\F_{11^2}(x,y,z)=\F_{11^2}(x,y)$. We conclude that $\F_{11^2}(x,y,z)/\F_{11^2}(x,y)$ is an unramified extension of degree two. Next, we claim that $\F_{11^2}$ is the full constant field of $G$. From the previous analysis, it is easily seen that the zeroes of $x^{12}-1$ give rise to $24$ rational places of $\F_{11^2}(x,y)$. Using the defining equations, one obtains that these zeroes are all zeroes of $y^2-1$ as well. Now consider one of these places $Q_{11}$, namely the common zero of $x-1$ and $y-1$. Then using Kummer's lemma, we see that there are two rational places lying above $Q_{11}$ corresponding to the two solutions for $z$ in $\F_{11^2}$ of the reduced equation $z^2=(\mi+1)f(1)f(\mi)=-1$. This shows that $G$ has full constant field $\F_{11^2}$. Now using that the genus of $\F_{11^2}(x,y)$ is nine, the Riemann--Hurwitz formula immediately implies that $G=\F_{11^2}(x,y,z)$ has genus $17$.
\end{proof}

\begin{lemma}\label{lem:N1G}
The function field $G$ has $496$ $\F_{11^2}$-rational places. 
\end{lemma}
\begin{proof}
We claim that the extension $\F_{11^2}(x,y,z)/\F_{11^2}(x)$ is an elementary abelian Galois extension of degree eight. First of all, we claim that the extension $\F_{11^2}(x,y,z)/\F_{11^2}(x)$ is an elementary abelian Galois extension of degree four. The four elements of the corresponding Galois group are the maps determined by $y \mapsto \pm y$ and $y \mapsto \pm x^{6}/y$. To see this, note for example that
\begin{eqnarray*}
\left(\left(\frac{x^{6}}{y}\right)^2-1\right)^2 & = & \frac{y^4-2x^{12}y^2+x^{24}}{y^4}\\
& = & \frac{2y^2-x^{12}-2x^{12}y^2+x^{24}}{2y^2-x^{12}}\\ & = & 1 - x^{12}.
\end{eqnarray*}
Since $\F_{11^2}(x,y,z)$ is the compositum of the Galois extensions $\F_{11^2}(x,y)$ and $\F_{11^2}(x,z)$ and we have seen in the proof of Lemma \ref{lem:genusG} that $[\F_{11^2}(x,y,z):\F_{11^2}(x)]=8$, the claim that ½$\F_{11^2}(x,y,z)/\F_{11^2}(x)$ is an elementary abelian Galois extension of degree eight follows. 

Using Galois theory, we then see that $G$ can be written as the compositum of three quadratic extensions of $\F_{11^2}(x)$, namely $\F_{11^2}(x,z)$, $\F_{11^2}(x,y+x^6/y)$ (the fixed field of the map $y \mapsto x^6/y$) and $\F_{11^2}(x,y-x^6/y)$ (the fixed field of the map $y \mapsto -x^6/y$). Writing $t=y+x^6/y$ and $u=y-x^6/y$, one can easily check that $t^2=2(1+x^6)$ and $u^2=2(1-x^6)$. Note that $1+x^6=f(\mi x)f(-\mi x)$ and $1-x^6=f(x)f(-x)$. Hence $G = \F_{11^2}(x,t,u,z)$, where
$$t^2=2 f(\mi x)f(-\mi x), \quad u^2=2 f(x)f(-x), \quad z^2=(\mi+1) f(x)f(\mi x).$$

In Lemma \ref{lem:genusG} we have already determined all ramification in $\F_{11^2}(x,y,z)/\F_{11^2}(x)$: only the zeroes of $x^{12}-1$ ramify and they do so with ramification index two. We claim that the resulting $12 \cdot 4=48$ places of $G$ lying above these zeroes are all rational. Now let $P$ be a place of $\F_{11^2}(x)$ that ramifies in $G/\F_{11^2}(x)$. Since $1-x^{12}=f(x)f(\mi x)f(-x)f(-\mi x)$, we distinguish two cases.

\bigskip
\noindent
{\bf Case 1} If $f(-x(P))=0$ or $f(-\mi x(P))=0$, then $P$ is unramified in exactly two of the three extensions $\F_{11^2}(x,t)/\F_{11^2}(x)$, $\F_{11^2}(x,u)/\F_{11^2}(x)$ and $\F_{11^2}(x,z)/\F_{11^2}(x)$. In fact, one can show directly that in the two unramified cases, $P$ splits. This amounts to checking that two out of the three values $2 f(\mi x(P))f(-\mi x(P))$, $2 f(x(P))f(-x(P))$ and $(\mi+1) f(x(P))f(\mi x(P))$ are nonzero squares in $\F_{11^2}.$
This means by \cite[Theorem 3.8.3]{stichtenothAlgebraicFunctionFields2009} that the decomposition field is an index two subfield of $G$. But since the ramification index of $P$ in $G$ is two, this means that there is no inertia. Hence any place of $G$ lying above $P$ is rational. 

\bigskip
\noindent
{\bf Case 2} If $f(x(P))=0$ or $f(\mi x(P))=0$, one proceeds similarly, after replacing the extension $\F_{11^2}(x,z)/\F_{11^2}(x)$ with $\F_{11^2}(x,tu/z)/\F_{11^2}(x)$. Note that $(tu/z)^2=4(\mi+1)^{-1}f(-x)f(-\mi x)$.

\bigskip

Next we consider rational places $P$ of $\F_{11^2}(x)$ that split completely in $G/\F_{11^2}(x)$. This occurs precisely if $P$ is a rational place such that $2(1+x(P)^6)$, $2(1-x(P)^6)$ and $(\mi+1)f(x(P))f(\mi x(P))$ are all nonzero squares in $\F_{11^2}$. A computer check quickly reveals that there are exactly $56$ such possible values for $x(P)$. This gives rise to $56 \cdot 8=448$ rational places of $G$. Adding the $48$ rational places lying above zeroes of $x^{12}-1$, we conclude that $G$ has at least $496$ rational places. Since the Hasse--Weil bound does not permit more, the lemma follows.
\end{proof}

\begin{theorem}
The function field $G$ is maximal.
\end{theorem}
\begin{proof}
Follows directly from Lemmas \ref{lem:genusG} and \ref{lem:N1G}.
\end{proof}

\begin{corollary}
    The Frobenius dimension of $G$ is $3$.
\end{corollary}

\begin{proof}
    This follows directly from Castelnuovo's bound for maximal function fields (see \cite[Corollary 10.25]{hirschfeldAlgebraicCurvesFinite2008}).
\end{proof}

\section{Not a subfield of the Hermitian function field}

In this section we show that the function field $G$ is not isomorphic to a subfield of the Hermitian function field $H_{11}$. Hence $G$ is the first example of a maximal function field with constant field $\F_{p^2}$ that is not covered by $H_p$. 

\begin{theorem}
The function field $G$ is not isomorphic to a subfield of the Hermitian function field $H_{11}$. 
\end{theorem}
\begin{proof}
Assume that $G$ is isomorphic to a subfield of the Hermitian function field $H_{11}$. In that case we can assume without loss of generality that $G$ is contained in $H_{11}$. We will prove the theorem by deriving a contradiction from this assumption. 

The genus of $H_{11}$ equals $11(11-1)/2=55$. Hence the Riemann--Hurwitz formula implies that $108 \ge [H_{11}:G]32$, whence $[H_{11}:G] \le \lfloor 108/32\rfloor=3$. On the other hand, since $G$ has exactly $496$ rational places, which at best all split completely in the extension $H_{11}/G$, we see that $H_{11}$ can have at most $[H_{11}:G] 496$ many places. On the other hand, it is well known that $H_{11}$ has exactly $11^3+1=1332$ many rational places. This implies that $[H_{11}:G] \ge \lceil 1332/496\rceil=3$. The above shows that $[H_{11}:G]=3$. We now distinguish two cases.

\bigskip
\noindent
{\bf Case 1: $H_{11}/G$ is a Galois extension} All Galois subfields of $H_{11}$ of index three have been classified, see for example \cite[Theorem 2.1]{CosKorTor}. For $p=11$, these give rise to maximal function fields of $H_{11}$ of genus $15$, $18$ and $19$. Hence $G$ is not among them and we may conclude that $G$ is not a Galois subfield of $H_{11}$ of index three.

\bigskip
\noindent
{\bf Case 2: $H_{11}/G$ is not a Galois extension} In this case the Galois closure of $H_{11}/G$ is a degree six extension of $G$ with Galois group $S_3$, the symmetric group on three letters. Given a place $P$ og $G$, there are two possible nontrivial ramification patterns: in the first place, $P$ can have two places lying above it (one with ramification index two and one with ramification index one) and in the second place, $P$ can be totally ramified. We denote the set of places where the first (resp. second) situation occurs by $N_2$ (resp. $N_3$). Further, we define $n_2=\sum_{P \in N_2} \deg(P)$ and $n_3=\sum_{P \in N_3} \deg(P)$.
Now note that on the one hand, the Riemann--Hurwitz formula implies that $D$, the degree of the different of the extension $H_{11}/G$, equals $D=108-3\cdot 32=12$, while by definition of $n_2$ and $n_3$ we find $D=n_2+2n_3$. Hence $n_2+2n_3=12$.

Let us denote the Galois closure of $H_{11}/G$ by $C$. Then $C$ is the composite of $H_{11}$ and the image of $H_{11}$ in $C$ under an involution $\sigma$ from the Galois group $S_3$. 
Now $P \in N_3$ and denote by $R$ a place of $C$ lying above it. Further write $Q=R\cap G$ and $Q'= R \cap \sigma(G)$. Since $P \in N_3$, we see that $e(Q|P)=e(Q'|P)=3$. Hence Abhyankar's lemma implies that $e(R|P)=3$. Moreover, we see that $Q$ is unramified in the extension $C/H_{11}$. Hence for a given $P\in N_3$, the contribution to the different degree of $C/G$ coming from the places lying above $P$ is precisely $4$.
Now consider $P \in N_2$. By definition of $N_2$, there exist two places $Q_1$ and $Q_2$ of $H_{11}$, necessarily rational, lying above $P$. Without loss of generality we may assume that $e(Q_1|P)=2$ and $e(Q_2|P)=1$. Denote by $S_1$ (resp. $S_2$) a place of $C$ lying above $Q_1$ (resp. $S_2$). Since $C/G$ is a Galois extension with Galois group the symmetric group on three letters, we see that $e(S_1|P)=e(S_2|P)=2$. Since we also can conclude that $e(S_1|Q_1)=1$, we derive that either $\deg(S_1)=2$ or that there are two places of $C$ lying above $Q_1$. In either case, for $P \in N_2$, the contribution to the different degree of $C/G$ coming from the places lying above $P$ is precisely $3$.

We are now ready to compute the genus of $C$ from the Riemann--Hurwitz formula. The result is: $g(C)=1+3\cdot 32+(3n_2+4n_2)/2$. Using that $n_2+2n_3=12$, we also obtain that $$g(C)=109+n_2/2=115-n_3$$
and hence $109 \le g(C) \le 115$.
Applying the Ihara bound, we see that $N_1(C) \le 2272$. 

Next we estimate the number of rational places of $C$.  There are exactly two places of $H_{11}$ lying above a place $P \in N_2$ and exactly one place of $H_{11}$ lying above $P \in N_3$. Hence there are at least $1332-2n_2-n_3$ rational places $Q$ of $H_{11}$ such that $P:=Q \cap G \not\in N_2 \cap N_3$. Since the Hermitian function field has no places of degree two and $Q$ is a rational place lying above $P$, the place $P$ splits completely in the extension $H_{11}/G$. Then the same is true of $P$ in any of the conjugates of $H_{11}$ in $C$ under the Galois group of $C/G$. Since the compositum of these conjugates is equal to $C$, we conclude that $P$ splits completely in $C/G$. In particular, the rational place $Q$ we started with splits completely in the extension $C/H_{11}$. But then $C$ contains at least 
\begin{eqnarray*}
2\cdot(1332-2n_2-n_3) & = & 2664-4n_2-2n_3\\
 & = & 2616+6n_3\\
 & \ge & 2616
\end{eqnarray*}
rational places, in contradiction with the Ihara bound. 
\end{proof}

\section{The automorphism group}

Write $A := \aut_{\overline{\mathbb{F}}_{11^2}}(G)$ for the full automorphism group of $G$. Since $G$ is maximal we know that $A$ is defined over $\mathbb{F}_{11^2}$ (see \cite[Theorem 3.10]{gunbyIrreducibleCanonicalRepresentations2015}). The aim in this section is to show that $|A|=192=12(g-1)$.
    
We start by constructing an explicit subgroup $A_0$ of automorphisms of $G$ of order $12(g-1)$, and later prove that in fact it is the full automorphism group. 

To do so, consider the following three involutions:

\begin{align*}
 \epsilon_1(x,y,z)&=(x,-y,z),\\
 \epsilon_2(x,y,z)&=\left(x,\frac{x^6}{y},z\right),\\
 \epsilon_3(x,y,z)&=(x,y,-z).
\end{align*}

A direct computation of their reciprocal compositions show that the three invlutions above commute, implying that 
\[
 H=\langle\epsilon_1,\epsilon_2,\epsilon_3\rangle
   \cong (C_2)^3.
\]

Two further automorphisms of $G$ can be defined as,
\begin{equation}\label{eq:tau}
 \begin{aligned}
 \tau(x)&=\frac1x,\\
 \tau(y)&=
 \frac{(\alpha-1)y-(\alpha+1)x^6/y}{2x^3},\\
 \tau(z)&=
 \frac{f(x)(y+x^6/y)}{x^3z},
 \end{aligned}
\end{equation}
and defining
\[
 \delta:=(7+4\alpha)x+1,
\]
one get another automorphism of the form
\begin{equation}\label{eq:sigma}
 \begin{aligned}
 \sigma(x)&=\frac{x+7+7\alpha}{\delta},\\
 \sigma(y)&=
 \frac{(y^2-1)/z+4\alpha z}{\delta^3},\\
 \sigma(z)&=
 \frac{4(\alpha+1)(y-x^6/y)}{\delta^3}.
 \end{aligned}
\end{equation}
Direct substitution into the two defining relations of the function field $G$ shows that in fact all the maps defined above give rise to automorphisms. Again a direct checking shows that
\[
 \tau^2=\epsilon_2\epsilon_3 \qquad \textrm{and} \
 \qquad ord(\sigma)=4.
\]
Note in particular that one has $ord(\tau)=4$.
We finally define
\[
A_0:=\langle\epsilon_1,\epsilon_2,\epsilon_3,\tau,\sigma\rangle.
\]
Our aim is to show that $|A_0|=192=12(g-1)$ and in particular that $A_0=((C_4 \times C_4) \rtimes C_3) \rtimes C_4$. A way to do that is to find generators for the cyclic groups of order $3$ and $4$ appearing in the claimed decomposition above, and prove that they normalize each other accordingly. 
Define
\begin{equation}\label{eq:polycyclic-generators}
 a=\epsilon_2\sigma^2,\qquad
 b=\tau\sigma^2\tau,\qquad
 c=\sigma\tau^{-1},\qquad
 d=\tau.
\end{equation}
It is easy to see that
\[
 ord(a)=ord(b)=ord(d)=4,\qquad ord(c)=3,
\]
and
\[
 \langle a,b\rangle\cong C_4\times C_4,\qquad
 |\langle a,b,c\rangle|=48,\qquad
 |\langle a,b,c,d\rangle|=192.
\]
Moreover,
\[
 \langle a,b\rangle\triangleleft\langle a,b,c\rangle
 \triangleleft G_0.
\]
Thus
\[
 A_0\cong ((C_4\times C_4)\rtimes C_3)\rtimes C_4,
\]
as claimed.

Before moving to the main theorem of this section, namely stating that $A=A_0$, we prove a technical lemma.

\begin{lemma} \label{transitive}
    $A_0$ acts transitively on the $48$ places of $G$ that ramify in the extension $G/\mathbb{F}_{11^2}(x)$.
\end{lemma}

\begin{proof}

We can prove the lemma by showing an explicit subgroup of $A_0$ that acts transitively
on the \(48\) zeroes of \(x^{12}=1\) in $G$. Recall that these are in fact exactly the $48$ ramified plaes in the extension $G/\mathbb{F}_{11^2}(x)$.

Consider the three involutions from $A_0$ generating an elementary abelian group of order $8$, that is,

\begin{align*}
 \epsilon_1(x,y,z)&=(x,-y,z),\\
 \epsilon_2(x,y,z)&=\left(x,\frac{x^6}{y},z\right),\\
 \epsilon_3(x,y,z)&=(x,y,-z)
\end{align*}

and denote with $B$ the set of $12$-th roots of unity in $\mathbb{F}_{11^2}$.
At every \(a\in B\), the ramification index of $(x=a)$ in $G/\mathbb{F}_{11^2}(x)$ is two, and hence in correspondence of each such $a$ we have four places in $G$.

Fix a place \(P_a\) above \(x=a\).  Since $\epsilon_i(x)=x$, the group \(H=\langle \epsilon_1, \epsilon_2, \epsilon_3 \rangle\) fixes \(x\), and hence acts on the four places above $x=a$.  The stabilizer \(H_{P_a}\) has order $2$, as $H$ has order $8$ and the ramification index of $(x=a)$ is two. Hence the length of the $H$-orbit containing $P_a$ is
\[
\frac{|H|}{|H_P|}=\frac82=4,
\]
implying that $H$ acts transitively on the places above $(x=a)$ for all $a \in B$.  
What is left to do is to find automorphisms of $A_0$ that act transitively on these different $H$-orbits of length $4$.  For this, only the actions on the
\(x\)-components of the automorphisms is needed. Consider the automorphisms $\tau$ and $\sigma$ defined before. One has
\begin{equation}\label{eq:base-actions}
 \tau(x)=\frac1x,
 \qquad
 \sigma(x)=\frac{x+7+7\alpha}{(7+4\alpha)x+1}.
\end{equation}
We prove that the above two $x$-actions preserve \(B\). Clearly $\tau$ preserves $B$ as if $a$ is a $12$-th root of unity so it $1/a$. Checking that the same is true for $\sigma$ note that

$$\bigg(\frac{x+7+7\alpha}{(7+4\alpha)x+1}\bigg)^{12}-1=\frac{(x+7+7\alpha)^{12}-((7+4\alpha)x+1)^{12}}{((7+4\alpha)x+1)^{12}}.$$

Looking at the numerator and using $\alpha^2=-1$ we have

$$(x+7+7\alpha)^{12}-((7+4\alpha)x+1)^{12}=$$
$$6\alpha^{12}x^{12} + 5\alpha^{12} + 5\alpha^{11}x^{12} + 7\alpha^{11}x^{11} + 7\alpha^{11} x + 5\alpha^{11} + 5\alpha x^{12}
    + 7\alpha x^{11} + 7\alpha x + 5\alpha + 7\alpha^{12} + 4$$
    $$=6x^{12} + 5 - 5\alpha x^{12} - 7\alpha x^{11} - 7\alpha x - 5\alpha + 5\alpha x^{12}
    + 7\alpha x^{11} + 7\alpha x + 5\alpha + 7 + 4$$
    $$=6x^{12} + 5=6(x^{12}-1).$$
    This shows that if $a \in B$ then also $\sigma(a) \in B$, as claimed.

Similarly by direct substitution, using
\(\alpha^2=-1\), gives the following three cycles for the decompositon of the action of
\(\sigma\) on $B$; each row is read cyclically from left to right:
\begin{align*}
 C_0&=(1,\;5-3\alpha,\;\alpha,\;-3+5\alpha),\\
 C_1&=(-1,\;3-5\alpha,\;-\alpha,\;-5+3\alpha),\\
 C_2&=(5+3\alpha,\;-3-5\alpha,\;3+5\alpha,\;-5-3\alpha).
\end{align*}
In fact these twelve entries are pairwise distinct and cover exatly the roots of
\(X^{12}-1\). 

The map $\tau$ now provides a way to join these $3$ orbits of $\sigma$ on $B$ as
\(\sigma\)-orbits:
\[
 \tau(\alpha)=\frac{1}{\alpha}=-\alpha,
 \qquad
 \tau(5-3\alpha)=\frac{1}{5-3\alpha}=5+3\alpha.
\]
As before in the identity above we used that $\alpha^2=-1$. Consequently
\begin{equation}\label{eq:base-transitive}
 \langle\tau,\sigma\rangle
 \quad\text{is transitive on the twelve elements of }B.
\end{equation}
Now we are ready to conclude our proof of the lemma. In fact let \(P,Q\) be any two of the \(48\) ramified places.  By
\eqref{eq:base-transitive}, there is an element $g \in \langle \tau,\sigma\rangle$
such that
\[
 x(g(P))=x(Q).
\]
Thus \(g(P)\) and \(Q\) lie in the same $A_0$ orbit. On the other hand, from the beginning of the proof we know that there exists
an \(h\in H\) such that \(h(g(P))=Q\).  Therefore $A_0$ acts transitively on all $48$ ramified places, as claimed.  
\end{proof}

\begin{lemma} \label{lem:HP48}
Let $P$ be a rational place and denote as usual by $L(18P)$ the Riemann--Roch space of the divisor $18P$. Then $\dim L(18P) \le 5$. Moreover, equality holds if and only if $P$ is unramified in the extension $G/\F_{11^2}(x)$. In particular, the automorphism group of $G$ acts on the set of $48$ rational places that ramify in $G/\F_{11^2}(x)$.
\end{lemma}
\begin{proof}
The statement about the dimension of the Riemann--Roch spaces can be verified directly using Magma. The magma code we used can be found in the appendix. The second part of the lemma now follows immediately, since for two rational places $P_1,P_2$ in the same orbit, the dimensions of $L(P_1)$ and $L(P_2)$ are the same.
\end{proof}

Our aim is now to show that 

\begin{theorem} \label{auto}
    The full automorphism group of $G$ is $A_0$. In particular $|A|=12(g-1)$.
\end{theorem}

\begin{proof}
    
    Write $|A| = 48 m$, where $m$ is the order of the stabilizer of a place in the short orbit of length 48. From $|G_0| = 192$ we know that $4\mid m$. Suppose for contradiction that $m>4$, i.e., $m\geq 4$.

    We first prove that $A$ is tame, that is $m$ is not divisible by $11$. In fact assume by contradiction that $A$ is non-tame and denote with $S$ a subgroup of order $11$. Since $A$ (and hence $S$) is defined over $\mathbb{F}_{11^2}$ \cite[Theorem 3.10]{gunbyIrreducibleCanonicalRepresentations2015}, $S$ acts on the set of places of degree one over $\mathbb{F}_{11^2}$. From Since by maximality the number of such places is congruent to $1$ modulo $11$, $S$ needs to fix a place. Furthermore since maximal curves have $p$-rank zero, from \cite[Lemma 11.131]{hirschfeldAlgebraicCurvesFinite2008} we have that $S$ has exactly one fixed place. Since neither $48$ (the size of our first short orbit) nor $11^2+1+2\cdot 17 \cdot 11-48$ is congruent to $1$ modulo $11$, $S$ cannot exist.

    This proves that the contributions to the Riemann--Hurwitz formula arising from subgroups of $A$ is always tame, fact that we will use in the following.

    The number of short orbits is at most three by point (I) and (II) in the proof of Theorem 11.56 in \cite{hirschfeldAlgebraicCurvesFinite2008}. By \cite[Lemma 11.111]{hirschfeldAlgebraicCurvesFinite2008} and the fact that we have a short orbit of size 48 (given by Lemmas \ref{lem:HP48} and \ref{transitive}), there have to be at least two short orbits. Suppose there are exactly two short orbits, one has order $48$ and we write $e$ for the order of the other one. The fixed field of $A$ is rational; in fact it is a subfield of $\mathbb{F}_{11^2}(x)$. Thus, Riemann--Hurwitz applied to $G/G^A$ implies 
    \begin{align*}
        32 = 2\cdot g(G) - 2 &= -2 \cdot 48m + 48(m-1) + n(48m/n - 1) \\
        &= -48 -n,
    \end{align*}
    which is a contradiction. We conclude that there are exactly three short orbits.

    Write $e_1$, $e_2$, $e_3$ for the orders of the short orbits. We know one of them is 48, say $e_1=48$. Applying Riemann-Hurwitz now gives
    $$
        32 = 48m \left( 1 - \frac{1}{e_2} - \frac{1}{e_3}\right) - 48,
    $$
    which can be rewritten as
    $$
        \frac{1}{e_2}+\frac{1}{e_3} = 1 - \frac{5}{3m}.
    $$

    If $e_2,e_3 \geq 3$ then the left hand side is at most $\frac{2}{3}$, but 
    $$
        1- \frac{5}{3m} \leq 1 - \frac{5}{24} < \frac{2}{3},
    $$
    so we may assume $e_2 = 2$. 

    Then, 
    $$
        \frac{1}{e_3} = \frac{1}{2} - \frac{5}{3m} \Longrightarrow e_3 = 2 + \frac{20}{3m-10},
    $$
    which implies that $3m -10$ divides $20$. It is clear that this does not hold for $m=8$, and for $m\geq 12$ we have $3m - 10 > 20$, so we arrive at a contradiction. 

    Combining the above shows $A = G_0$ as wished.
\end{proof}

\begin{remark}
From \cite{BMT} every $\mathbb{F}_{p^2}$-maximal function field of genus $g \geq 2$ whose automorphism group has order greater than $84(g-1)$ is a Galois subcover of the Hermitian function field. In the paper an example of an $\mathbb{F}_{71^2}$-maximal function field that is not a Galois subfield of the Hermitian function field $H_{p}$ is provided, though the subcover (not Galois) problem is left open. The function field $G$ is the first know $\mathbb{F}_{p^2}$-maximal function field that is not a subcover of the Hermitian function field. Note that by Theorem \ref{auto} the order of the automorphism group of $G$ is $12(g-1)$, which is quite close to the $84(g-1)$ that would imply automatically the function field being a subfield of $H_p$.
\end{remark}

\section{Acknowledgments and AI tool disclosure}
This work was supported by a research grant (VIL''52303'') from Villum Fonden. This work was supported in part by a grant of access to OpenAI models through the
ChatGPT for Academic Researchers program.

GPT 5.6 Sol was used to find the function field presented in this paper. GPT 6 Astra and GPT 5.6 were used for the work that let to determining the full automorphism group. Apart from this, everything in the paper was human generated.

\bibliographystyle{abbrv}
\bibliography{ref}

\newpage

\appendix

\section{Magma code}

\begin{verbatim}
q := 11;
G<r> := GF(q^2);
F<x> := FunctionField(G);
P<w> := PolynomialRing(F);
F<y> := ext< F | (w^2-1)^2 - 1 + x^12>;

a := r^30;

P<w> := PolynomialRing(F);
A := (x+1)*(x^2+x+1);
B := (a*x+1)*((a*x)^2+a*x+1);
F<z> := ext< F | w^2 - (a+1)*A*B>;

check := true;

for P in Places(F,1) do 
    if IsOne(Evaluate(x,P)^12) then
        if not IsZero(Dimension(18*P)-5) then 
            "PROBLEM", P; check := false; break P;
        end if;
    else
        if not IsZero(Dimension(18*P)-4) then
            "PROBLEM", P; check := false; break P;
        end if;
    end if;
end for;

check;
\end{verbatim}

\bigskip 

The above code was executed and gave the output ``true'' with Magma V2.29-10 via \url{https://magma.maths.usyd.edu.au/calc/}
Seed: 2639617556; Total time: 42.020 seconds; Total memory usage: 32.09MB.

\end{document}